\pdfoutput=1
\RequirePackage{ifpdf}
\ifpdf 
\documentclass[pdftex]{sigma}
\else
\documentclass{sigma}
\fi

\numberwithin{equation}{section}

\newtheorem{Theorem}{Theorem}[section]
\newtheorem{Corollary}[Theorem]{Corollary}
\newtheorem{Lemma}[Theorem]{Lemma}
 { \theoremstyle{definition}
\newtheorem{Remark}[Theorem]{Remark} }

\newcommand{\bigzero}{\mbox{\normalfont\Large\bfseries *}}

\begin{document}

\allowdisplaybreaks

\newcommand{\arXivNumber}{2007.02913}

\renewcommand{\PaperNumber}{133}

\FirstPageHeading

\ShortArticleName{Determinantal Expressions in Multi-Species TASEP}

\ArticleName{Determinantal Expressions in Multi-Species TASEP}

\Author{Jeffrey KUAN}

\AuthorNameForHeading{J.~Kuan}

\Address{Texas A\&M University, Department of Mathematics,\\ Mailstop 3368, College Station, TX 77843-3368, USA}
\Email{\href{mailto:jkuan@math.tamu.edu}{jkuan@math.tamu.edu}}
\URLaddress{\url{http://math.tamu.edu/~jkuan/}}

\ArticleDates{Received July 14, 2020, in final form December 03, 2020; Published online December 11, 2020}

\Abstract{Consider an inhomogeneous multi-species TASEP with drift to the left, and define a height function which equals the maximum species number to the left of a lattice site. For each fixed time, the multi-point distributions of these height functions have a~determinantal structure. In the homogeneous case and for certain initial conditions, the fluctuations of the height function converge to Gaussian random variables in the large-time limit. The proof utilizes a coupling between the multi-species TASEP and a coalescing random walk, and previously known results for coalescing random walks.}


\Classification{60J27}

\section{Introduction}
Consider a multi-species TASEP (totally asymmetric simple exclusion process) on the infinite lattice $\mathbb{Z}$. The species numbers are indexed by the integers, so any particle configuration can be expressed as $\eta(t) = (\eta_x(t))_{x\in \mathbb{Z}}$, where $\eta_x(t)$ denotes the species number of the particle at site~$x$. By convention, let $\eta_x(t)=-\infty$ denote a hole (no particle) at lattice site~$x$. The particle at site~$x$ jumps one step to the left with exponential jump rates of rate $\lambda_x$, and particles with higher species have priority to jump. In other words, if a particle of species~$j$ attempts to jump to a site of species $i$, then the particles switch places if $i<j$, and the jump is blocked if $i>j$. All jumps occur independently of each other. For $\mu \in \mathbb{Z}\cup \{\pm \infty\}$, let $\xi_{\mu}(t)$ denote the location of a single random walker which starts at $\mu$, and jumps one step to the left with exponential rate $\lambda_x$ when at site $x$. If $\mu= \{\pm \infty\}$ then the particle remains at $\{\pm \infty\}$ for all times $t \geq 0$. In other words, $\xi_{\mu}(t)$ describes the evolution of the single-species inhomogeneous TASEP with only a single particle starting at lattice site~$\mu$.

Define
\begin{equation*}
g_y(t) = \max \{ \eta_x(t)\colon x< y\} .
\end{equation*}
In words, this is the maximum species number among all particles to the left of $y$.
The main result of this paper is the following theorem.

\begin{Theorem}\label{Theorem1.1}
Consider any initial conditions $\eta_x(0)$. Fix $k$ lattice points $x_1<\dots<x_k$ and $k$ species numbers $m_1<\dots< m_k$, and let
\begin{equation*}
\mu_i = \inf \{x\colon \eta_x(0) \geq m_i\},
\end{equation*}
where $\inf \varnothing = +\infty$ by convention. Then
\begin{equation*}
\mathbb{P} ( g_{x_1}(t) \geq m_1, \dots, g_{x_k}(t) \geq m_k ) = \det[G_{ij}]_{1\leq i, j \leq k},
\end{equation*}
where $G$ is the $k\times k$ matrix with entries
\begin{equation*}
G_{ij} = \mathbb{P}( \xi_{\mu_i}(t) < x_j ) - 1_{\{ {i < j} \}}.
\end{equation*}
\end{Theorem}

In the homogeneous case when all $\lambda_x \equiv 1$, it is readily seen (by the central limit theorem) that the fluctuations are Gaussian.
\begin{Corollary}\label{Corollary1}
Suppose that all $\lambda_x$ are equal to $1$. Assume that $x_1,\dots,x_k$ and $m_1,\dots,m_k$ depend on $t$ in such a way that
\begin{equation*} \frac{\mu_i - \nu_i t}{t^{1/2}} \rightarrow w_i , \qquad \frac{x_j - \nu_j't}{t^{1/2}}\rightarrow w_j'.\end{equation*}
for some $\nu_i,w_i,\nu_j',w_j' \in \mathbb{R}$.
Then
\begin{equation*}
\lim_{t\rightarrow \infty} \mathbb{P} ( g_{x_1}(t) \geq m_1, \dots, g_{x_k}(t) \geq m_k ) = \det[M_{ij}]_{1\leq i, j\leq k},
\end{equation*}
where
$M$ has entries
\begin{equation*}
M_{ij} = \Phi(w_i - w_j') - 1_{\{i<j\}}.
\end{equation*}
Here, $\Phi$ is the cumulative distribution of a standard Gaussian
\begin{equation*}
\Phi(z) = \int_{-\infty}^z \frac{1}{\sqrt{2 \pi }} {\rm e}^{-y^2/(2)}{\rm d}y.
\end{equation*}
\end{Corollary}

\begin{Remark}\label{Remark1}
If $m_i \geq m_j$ for some $i<j$ in Theorem~\ref{Theorem1.1}, then
\begin{equation*}
\{ g_{x_i}(t) \geq m_i \} \subseteq \{ g_{x_j}(t) \geq m_j\},
\end{equation*}
so there is no loss of generality in assuming that $m_1< \dots < m_k$. In other words, if $m_i \geq m_j$ and $i<j$, there is the equality
\begin{equation*}
\mathbb{P}\bigg( \bigcap_{a \in \{1,\dots,k\}} \{g_{x_a}(t) \geq m_a\}\bigg) = \mathbb{P}\bigg( \bigcap_{a \in \{1,\dots,k\}-\{j\} }\{g_{x_a}(t) \geq m_a\}\bigg),
\end{equation*}
so one can successively remove indices until what remains is an increasing sequence
\[m_{a_1} < \dots < m_{a_l}.
\]
\end{Remark}

\begin{Remark}\label{Remark2}
Since $m_1<\dots<m_k$, then
\begin{equation*}
\{x\colon \eta_x(0) \geq m_1\} \supseteq \cdots \supseteq \{x\colon \eta_x(0) \geq m_k\},
\end{equation*}
so $\mu_1 \leq \dots \leq \mu_k$. If $\mu_i=\infty$ for some $i$, then $\mu_k=\infty$, so the $k$th row of $G$ consists entirely of zeroes. This is consistent with $\mathbb{P}(h_{x_k}(t) \geq \infty)=0$. Similarly, if $\mu_i=-\infty$ for some $i$, then $\mu_l=-\infty$ for $1\leq l \leq i$, so the $l$th row of $G$ equals $(\underbrace{1,\dots,1}_{l\ \text{times }},0,0,\dots,0)$ for $1 \leq l \leq i$. Then the matrix $G$ is a block matrix of the form
\begin{equation*}
\left(\begin{array}{@{}c|c@{}}
 \begin{matrix}
 1 & 0 & \cdots & 0 \\
 1 & 1 & \cdots & 0 \\
 \vdots & \vdots& \ddots & 0\\
 1 & 1 & \cdots & 1
 \end{matrix}
 &\bigzero \\
\hline
\bigzero &
 \begin{matrix}
 \bigzero
 \end{matrix}
\end{array}\right),
\end{equation*} so its determinant equals the determinant of the lower right $(k-l)\times (k-l)$ block. This is consistent with $\mathbb{P}(h_{x_1}(t) \geq -\infty,\dots, h_{x_i}(t) \geq -\infty )=1$.
\end{Remark}

\begin{Remark}\label{Remark3}
Two initial conditions which satisfy the assumptions of Corollary~\ref{Corollary1} are the step initial conditions
\begin{equation*}
\eta_x(t) = x \qquad \text{for} \quad x \in \mathbb{Z},
\end{equation*}
and the flat initial conditions
\begin{equation*}
\eta_x(t) =
\begin{cases}
x/2 & \text{for } x \text{ even},\\
-\infty & \text{for } x \text{ odd}.
\end{cases}
\end{equation*}
In these two cases, respectively, $\mu_i=m_i$ and $\mu_i=2m_i$. Then
\begin{equation*}
\{g_{x_i}(t) \geq m_i\} = \left\{ \frac{g_{x_i}(t) - \nu_i t}{t^{1/2}} \geq w_i \right\} \qquad \text{or} \qquad \left\{ \frac{2g_{x_i}(t) - \nu_i t }{t^{1/2}} \geq w_i \right\}.
\end{equation*}
\end{Remark}

\begin{Remark}\label{Remark4}
For the (single-species) TASEP with various initial conditions (step, flat, stationary, and their mixtures), it has been shown in \cite{BFP10,BFP07,bfps1,BFS08a,BFS08b, JohPNG} that the fluctuations in the usual height function converges to various Airy processes; see also~\cite{quastel2019kp}. Intuitively, these fluctuations are asymptotically independent of the fluctuations of $h_{x_i}$ in Corollary~\ref{Corollary1}. For example, letting~$h_x(t)$ denote the number of particles to the left of $x$ at time~$t$, then with step initial conditions
\begin{equation*}
\lim_{t\rightarrow \infty} \mathbb{P}\left( \frac{h_x(t)-a_1t}{a_2t^{1/3}} \geq -s, \frac{g_x(t)-b_1t}{b_2t^{1/2}} \leq s'\right) = F_2(s)\Phi(s'),
\end{equation*}
where $F_2$ is the Tracy--Widom distribution. The position $x$ depends on $t$ as $x=-\nu t$ for $\nu \in (0,1)$, and the constants are given by $a_1=\tfrac{1}{4}(1-\nu)^2$, $a_2 = 2^{-4/3}\big(1-\nu^2\big)^{2/3},$ and $b_1=1-\nu$, $b_2=(1-\nu)^{1/2}$. To see independence, simply note that the fluctuations of $h_x(t)$ depend on the exponential clocks of the particles starting with distance less than $a_1t + O\big(t^{1/3}\big)$ of $x$, whereas the fluctuations of $g_x(t)$ depend on the exponential clocks of the particles starting with distance $b_1t + O\big(t^{1/2}\big)$ from $x$. (A rigorous proof requires a quantification of the statement that far-away clocks have negligible contribution to the fluctuations). These sets of clocks are independent of each other. Note that an independent Tracy--Widom and Gaussian also appear in the asymptotic crossing probability of the AHR model~\cite{CdGHS}.
\end{Remark}
\begin{Remark}\label{Remark5}
Determinantal expressions in the multi-species TASEP have appeared before in different contexts: \cite{ChaSch,ELee}, see also~\cite{SchMaster}.
\end{Remark}

\section{Proof of Theorem~\ref{Theorem1.1}}
The theorem is proved by coupling the multi-species TASEP to a coalescing random walk, and using a known determinantal expression for the distributions of coalescing random walk from Proposition~2.5 of~\cite{assiotis2018} (see also Proposition~9 of~\cite{Warren} for a similar expression in the context of coalescing Brownian motions). This coupling does not seem to have appeared in the literature so far; but note the upcoming preprint~\cite{BBUp}, in which a similar coupling is being used.

Before describing the coupling, we first need to modify the initial conditions $\eta_x(0)$. There may be values of $x$ and $y$ where $x<y$ and $\eta_x(0)> \eta_y(0)$. In other words, this means that the particle starting at~$x$ has higher species number and is to the left of the particle starting at~$y$. Thus, the particle starting at~$y$ can never contribute to any value of~$g_z(t)$ for $z \in \mathbb{Z}$ and $t\geq 0$, because any contribution made by the particle starting at $y$ has already been made by the particle starting at~$x$. Thus, removing the particle at~$y$ in the initial conditions will not change the joint distribution of~$g_{x_i}(t)$. Therefore we may assume that the particles in the initial conditions are ordered by species number.

In the coalescing random walk, if a particle tries to jump to a site that is already occupied, then the two particles coalesce into a single particle. The particle at site $x$ has an exponential clock of rate $\lambda_x$, and jumps one step to the left when the clock rings, just as in the multi-species TASEP. The coupling is constructed as follows. The coalescing (multi-species) random walk and the multi-species TASEP begin with the same initial conditions. We couple the exponential clocks of each particle in the multi-species TASEP with the exponential clocks of the corresponding particle in the coalescing random walk. In this way, each particle in the multi-species TASEP is coupled with a particle in the coalescing random walk at time $t=0$. The evolution occurs with the following rules:
\begin{enumerate}\itemsep=0pt
\item
If a particle in the multi-species TASEP attempts to jump but the jump is blocked, then no change occurs in either system.
\item
If a particle in the multi-species TASEP jumps one step to the left and that site is unoccupied, then the corresponding coupled particle (if there is one) in the coalescing random walk makes the same jump.
\item
If a particle in the multi-species TASEP jumps one step to the left, and that site is occupied by a lower species particle, then the corresponding coupled particle (if there is one) makes the same jump in the coalescing random walk. Additionally, the lower species particle is then absorbed into the higher species particle in the coalescing random walk (if it had not already been absorbed previously), and is no longer coupled with the corresponding particle in the multi-species TASEP.
\item
If a particle in the multi-species TASEP attempts a jump but there is no coupled particle in the coalescing random walk, then no jump occurs in the coalescing random walk.
\end{enumerate}
One can see that two particles are coupled if and only if they occupy the same location. Below is an example of the particle updates illustrating each of the four rules. The color red is used to indicate that particle attempts to jump to the left:
\begin{gather*}
\underline{\ }\ \underline{1}\ \underline{{\color{red}2}}\ \underline{3}\ \underline{4}\ \stackrel{\text{rule 3}}{\longrightarrow}\underline{\ } \ \underline{{\color{black}2}}\ \underline{{\color{red}1}}\ \underline{3}\ \underline{4}\ \stackrel{\text{rule 1}}{\longrightarrow} \underline{\ } \ \underline{{\color{red}2}}\ \underline{1}\ \underline{3}\ \underline{4}\ \stackrel{\text{rule 2}}{\longrightarrow} \underline{2} \ \underline{\ }\ \underline{{\color{red}1}} \ \underline{3}\ \underline{4} \stackrel{\text{rule 4}}{\longrightarrow} \underline{2} \ \underline{1 }\ \underline{\ } \ \underline{3}\ \underline{4}\,, \\
\underline{\ }\ \underline{1}\ \underline{{\color{red}2}}\ \underline{3}\ \underline{4}\ \stackrel{\text{rule 3}}{\longrightarrow} \underline{\ } \ \underline{2}\ \underline{{\color{white} 1} }\ \underline{3}\ \underline{4} \ \stackrel{\text{rule 1}}{\longrightarrow} \underline{\ } \ \underline{{\color{red}2}}\ \underline{{\color{white} 1} }\ \underline{3}\ \underline{4}\ \stackrel{\text{rule 2}}{\longrightarrow} \underline{2} \ \underline{\ }\ \underline{{\color{white} 1}} \ \underline{3}\ \underline{4} \stackrel{\text{rule 4}}{\longrightarrow} \underline{2} \ \underline{\ }\ \underline{{\color{white} 1}} \ \underline{3}\ \underline{4}\,.
\end{gather*}
Also note that there is no rule to couple the system after updates of the form
\begin{gather*}
\underline{2} \ \underline{{\color{red}1 } } \longrightarrow \underline{2} \ \underline{1}\,,\\
\underline{\circ}\ \underline{{\color{red}{\circ}}} \longrightarrow \underline{\circ} \ \underline{\ } \,.
\end{gather*}
However, the initial conditions preclude these configurations from occurring, because the par\-ticles are ordered by species number.

Let $\tilde{\eta}_x(t)$ denote the species of the particle located at site $x$ at time $t$ in the coalescing random walk. As before, let $\tilde{\eta}_x(t)=-\infty$ by convention if there is no particle at site $x$ and time $t$. Now let
\begin{gather*}
\tilde{g}_y(t) = \max \{ \tilde{\eta}_x(t)\colon x<y\}.
\end{gather*}

\begin{Lemma}\label{Lemma2.1}
For $x_1<\dots <x_k$ and any fixed time $t$, the joint distribution of $(g_{x_1}(t),\dots,g_{x_k}(t))$ equals the joint distribution of $(\tilde{g}_{x_1}(t),\dots, \tilde{g}_{x_k}(t))$. In other words, for any $m_1<\dots<m_k$,
\begin{equation}\label{Same}
\mathbb{P} ( g_{x_1}(t) \geq m_1, \dots, g_{x_k}(t) \geq m_k ) = \mathbb{P} ( \tilde{g}_{x_1}(t) \geq m_1, \dots,\tilde{g}_{x_k}(t) \geq m_k ).
\end{equation}
\end{Lemma}

\begin{proof}At time $t=0$, the multi-species TASEP and the coalescing random walk are in the same configuration, so \eqref{Same} holds. The only potential way that the two sides of \eqref{Same} could be different is when a particle jump causes the two processes to be updated in a different manner~-- that is, update rules~3 or~4.

Suppose that an update under rule~4 occurs. This means that a particle has jumped in multi-species TASEP when there is no coupled particle in the coalescing random walk. This means that the particle in the coalescing random walk has already been absorbed into a higher species particle. Correspondingly, this means that in the multi-species TASEP, there is a higher species particle to the left of that particle. This means that the evolution of the uncoupled/absorbed particle no longer affects the values of $g_{x_i}(t)$ or $\tilde{g}_{x_i}(t)$, since the particle cannot cross~$x_i$ unless the higher species particle has also crossed $x_i$. Similarly, if an update under rule~3 occurs, then the multi-species TASEP and the coalescing random walk differ in the uncoupled/absorbed particle. For identical reasons, the evolution of this particle does not affect the values of $g_{x_i}(t)$ or $\tilde{g}_{x_i}(t)$.
\end{proof}

By Lemma~\ref{Lemma2.1}, it suffices to find the right-hand side of~\eqref{Same}. As mentioned in Remark~\ref{Remark2}, $\mu_1 \leq \dots \leq \mu_k$, where $\mu_i$ is defined as in Theorem~\ref{Theorem1.1}. Let $\tilde{x}_{\mu}(t)$ denote the position of the particle in the coalescing random walk that started at lattice site $\mu$ at time~$0$. Then it is straightforward to see that there is the inclusion of events
\begin{equation*}
\{\tilde{g}_{x_i}(t) \geq m_i\} \supseteq \{\tilde{x}_{\mu_i}(t) \leq x_i\}.
\end{equation*}
To see the reverse inclusion, note that $\tilde{x}_{\mu_i}(t)$ is the left-most particle which has species higher than~$m_i$, and remains the left-most such particle after coalescence. Therefore,
\begin{equation*}
 \mathbb{P} ( \tilde{g}_{x_1}(t) \geq m_1, \dots ,\tilde{g}_{x_k}(t) \geq m_k ) = \mathbb{P}( \tilde{x}_{\mu_1}(t)\leq x_1, \dots, \tilde{x}_{\mu_k}(t) \leq x_k ).
\end{equation*}
Note that the random variables $\tilde{x}_{\mu}(t)$ on the right-hand side do not depend on the particles' species. The probability on the right-hand side can be found from Proposition~2.5 of~\cite{assiotis2018}. The cited proposition applies more generally (it allows for right jumps as well), and includes a ref\-lec\-ting boundary. We can take the limit of this boundary point to $-\infty$ and apply the proposition. The result is that the probability equals $\det[G_{ij}]$, as needed.

\subsection*{Acknowledgements}

The author thanks Alexei Borodin and Alexey Bufetov for helpful conversations.

\pdfbookmark[1]{References}{ref}
\LastPageEnding

\end{document}